\documentclass[letterpaper, 10 pt, conference]{ieeeconf}  % Comment this line out if you need a4paper

\IEEEoverridecommandlockouts                              % This command is only needed if
\usepackage{amsmath}
\usepackage{amssymb}
\usepackage{amsthm}
\usepackage{bm}
\usepackage{xcolor}
\usepackage[algo2e,ruled,linesnumbered]{algorithm2e}
\usepackage{graphicx}
\usepackage{booktabs}
\usepackage{multirow}
\usepackage{float}
\usepackage{url}   % or \usepackage{hyperref}

\definecolor{CommentBlue}{RGB}{70,130,180}

\SetAlgoNoLine
\SetAlgoLongEnd
\SetKwSty{textbf}
\SetArgSty{textnormal}
\SetCommentSty{mycommentfont}
\SetKwComment{Comment}{// }{}
\SetKwInOut{Input}{Input}
\SetKwInOut{Output}{Output}

\newtheorem{assumption}{Assumption}

\newtheorem{corollary}{Corollary}
\newtheorem{proposition}{Proposition}
\newtheorem{definition}{Definition}
\newtheorem{lemma}{Lemma}
\theoremstyle{remark}
\newtheorem{remark}{Remark}

\newcommand{\norm}[1]{{\left\vert\kern-0.25ex\left\vert #1\right\vert\kern-0.25ex\right\vert}}

\title{\LARGE \bf
Distributed Grid Search for Hyperparameter Tuning
}

\author{Constança De Herdeiro Cachim$^{1}$, Andrea Martin$^{1}$, Nicola Bastianello$^{2}$% <-this % stops a space
\thanks{$^{1}$KTH Royal Institute of Technology, Stockholm, Sweden {\tt\small ccachim@kth.se, andrmar@kth.se}}%
\thanks{$^{2}$University College London, United Kingdom {\tt\small n.bastianello@ucl.ac.uk}}%
}

\begin{document}

\maketitle
\thispagestyle{empty}
\pagestyle{empty}

\begin{abstract}
The performance of many distributed optimization algorithms over peer-to-peer networks strongly depends on the choice of their hyperparameters. However, existing selection rules often require global quantities, yield conservative choices, or are tied to specific algorithms and hyperparameters. We propose Distributed Grid Search (D-GS), a consensus-based search method that evaluates and selects hyperparameter configurations across the network, requiring only the ability to run the tuned algorithm and evaluate local gradients. We also propose a zoom-in variant, Distributed Zoom Grid Search (D-ZGS), which improves empirical performance at a reduced cost. We prove that both methods terminate in finite time with all agents holding the same hyperparameter configuration, and quantify the (sub)optimality of this configuration. Numerical experiments on linear and logistic regression show that D-GS closely matches a centralized reference and that tuning and deployment cost a fraction of deploying algorithms with hyperparameters tuned using theoretical results.
\end{abstract}

\section{Introduction}

Learning and optimization tasks are increasingly performed over networks of agents, driven by applications in areas such as smart grids, multi-robot systems, and signal processing \cite{molzahn_2017, nedic_2018, vlaski_2023}. 
In these systems, data is generated in multiple places at once, and collecting it all centrally is often impractical as it can entail significant communication costs,  raise privacy concerns and create a single point of failure \cite{meng_2021}.
Instead, decentralized approaches avoid these limitations by processing data where it is collected and collaborating over networks \cite{liu_2024}. These architectures include federated settings, where a central server coordinates \cite{li_2021}, and distributed settings, where agents communicate directly with each other \cite{yang_2019}.

Decentralized algorithms rely on hyperparameters, such as the stepsize, that are not learned by the algorithm itself. Their selection has a decisive impact on convergence speed and solution quality, with poor choices causing algorithms to converge slowly or even diverge  \cite{bengio_2012, rikos_2025, xu_2017}.
In practice, these values are commonly tuned by hand, through costly trial and error, often assuming access to global information (\textit{e.g.} knowledge of the network topology and agents' data). This motivates the design of automated, decentralized hyperparameter tuning methods, which only rely on local computation and communications. In particular, in this paper we focus on the distributed setting.

The literature on decentralized hyperparameter tuning is relatively scarce compared to centralized set-ups; examples include the tuning of stepsizes via decentralized version of the Barzilai--Borwein and Polyak rules \cite{hu_2025, fainman_2026}, and line searches coordinated through consensus protocols \cite{kuruzov_2025, chen_2025, rikos_2025}. Coordination is important in this set-up, as tuning methods yielding heterogeneous stepsizes across agents might lead to degraded performance or even divergence; indeed, most algorithms are designed to use coordinated hyperparameters.

While these works focus exclusively on stepsize tuning, many decentralized algorithms are characterized by several hyperparameters. This motivates the development of more general tuning algorithms based on hyperparameters optimization.
Hyperparameter optimization casts the selection of a hyperparameter configuration as an optimization problem \cite{franceschi_2025}. It has been studied most extensively in centralized settings, with decentralized adaptations emerging only more recently, mostly in federated settings \cite{khodak_2021,khan_2024}. 
In peer-to-peer networks, existing approaches based on decentralized bilevel optimization require the performance metric to be differentiable in the hyperparameters and reach agreement only asymptotically~\cite{dong, chen_2025a}. A method that guarantees finite-time, network-wide agreement, and that is applicable to tune any algorithm, is thus still missing.

\textit{Contributions:} The contributions of this paper are as follows.
\begin{itemize}
    \item We propose Distributed Grid Search (D-GS), a method that scores candidate hyperparameters by average consensus and reaches agreement by minimum consensus, requiring only the ability to run the tuned algorithm and evaluate local gradients. We also propose a zoom-in variant, Distributed Zoom Grid Search (D-ZGS), which improves empirical performance while reducing computational cost.
    
    \item We prove that both methods terminate in finite time with all agents holding the same configuration. We also quantify the (sub)optimality of the selected configuration, and its dependence on the accuracy of the averaging protocol, detailing three such protocols with different trade-offs.
    
    \item We validate D-GS and D-ZGS on linear and logistic regression problems, comparing the selected configurations against a centralized reference, studying the cost of tuning against a theoretical stepsize, and examining the effect of averaging errors on the selection.
\end{itemize}

\textit{Outline:} Section~\ref{sec:problem_formulation} formulates the hyperparameter selection problem and its challenges in a distributed setting. Section~\ref{sec:prelims} reviews the consensus protocols that the proposed algorithms build on. Section~\ref{sec:algorithm} presents D-GS and D-ZGS and establishes their guarantees. Section~\ref{sec:results} reports our numerical results, and Section~\ref{sec:conclusion} concludes the paper.

\textit{Notation:} $\lVert \cdot \rVert$ denotes the Euclidean norm of a vector, $\mathbf{1}$ the vector of all ones and $\lvert \cdot \rvert$ the cardinality of a set. For a graph $\mathcal{G} = (\mathcal{V}, \mathcal{E})$, $\mathcal{V} = \{1, \dots, N\}$, $\mathcal{N}_i = \{ j \in \mathcal{V} \mid \{i,j\} \in \mathcal{E} \}$ denotes the neighbors of agent $i$ and $D$ the diameter of $\mathcal{G}$.

\section{Problem Formulation}
\label{sec:problem_formulation}

Consider $N \in \mathbb{N}$ agents that communicate over an undirected and connected graph $\mathcal{G} = (\mathcal{V}, \mathcal{E})$, cooperating toward the solution of the optimization problem
\begin{equation}\label{eq:problem}
    \min_{x_1, \dots, x_N \in \mathbb{R}^n} \; \sum_{i=1}^{N} f_i(x_i) \quad \text{s.t.} \quad x_1 = \ldots = x_N,
\end{equation}
where $x_i \in \mathbb{R}^n$ is the state of agent $i$ and $f_i : \mathbb{R}^n \to \mathbb{R}$ represents its local cost, privately held by agent $i$.

Let $\mathcal{A}$ be any distributed algorithm applied by the agents to solve~\eqref{eq:problem}, characterized by $H$ hyperparameters (\textit{e.g.} stepsize, penalty, ...) collected in the configuration $\lambda \in \mathbb{R}^H$. We denote the output of $\mathcal{A}$ with configuration $\lambda$ and after $\tau$ iterations by $\bm{x}^{[\tau]} = \mathcal{A}(\bm{x}^{[0]}, \lambda, \tau)$, where $\bm{x}^{[0]} = (x_1^{[0]}, \dots, x_N^{[0]}) \in \mathbb{R}^{n N}$ are the initial conditions, and $x_i^{[t]}$ denotes the state of agent $i$ at iteration $t \in \mathbb{N}$. The performance of $\mathcal{A}$ -- for example, how quickly the $x_i^{[t]}$ converge to a solution of~\eqref{eq:problem} -- strongly depends on the choice of $\lambda$. The goal in this paper is to design \textit{distributed} algorithms that find the configuration leading to the best performance, as defined in the following.

Let $\Lambda = \Lambda_1 \times \cdots \times \Lambda_H \subset \mathbb{R}^H$ be the set of admissible hyperparameter configurations, with $\Lambda_h$ being bounded intervals.
Denoting by $J(\mathcal{A}(\bm{x}^{[0]}, \lambda, \tau)) \in \mathbb{R}$ the metric that evaluates the performance of $\mathcal{A}$, the objective is for the agents in $\mathcal{G}$ to minimize $J$ over the configuration space $\Lambda$.
The dependence of $J$ on hyperparameters $\lambda$ is indirect, since it is applied to the output of $\mathcal{A}$; nonetheless, in the following we write $J(\lambda)$ as shorthand for $J(\mathcal{A}(\bm{x}^{[0]}, \lambda, \tau))$. Therefore, we formulate the hyperparameter tuning problem of interest as:\looseness-1
\begin{equation}\label{eq:selection}
    \lambda^\star \in \operatorname*{arg\,min}_{\lambda \in \Lambda} \; J(\lambda).
\end{equation}

This problem is in general difficult to solve, since $J$ need not have a closed form expression and might be nonconvex. Additionally, solving~\eqref{eq:selection} in a distributed fashion raises several challenges, summarized in the following and addressed in section~\ref{sec:algorithm}:
\begin{itemize}
\item[\textbf{(C1)}] Selecting and evaluating $J$: there are several options for choosing $J$, including the empirical convergence rate or a measure of suboptimality. However, only a subset of these options is suitable for distributed hyperparameter tuning, since no individual agent has access to global information. The goal thus is to \textit{select a performance metric that can be evaluated (at least approximately) via a distributed algorithm}.

\item[\textbf{(C2)}] Coordinating the choice of $\lambda$: in a federated setting, the choice of which configuration $\lambda$ to employ after tuning can be performed centrally by the coordinator. In a distributed setting, however, \textit{we need to design a coordination protocol so that the agents can cooperatively pick which $\lambda$ to use}.
\end{itemize}
Section~\ref{sec:algorithm} details how the proposed tuning methods address these challenges using average consensus and minimum consensus algorithms as building blocks. In the next section, we present some preliminaries about these consensus algorithms.

\section{Preliminaries on Consensus Algorithms}
\label{sec:prelims}

\subsection{Average consensus}
\label{sec:avgcons}
An average consensus protocol allows the agents to cooperatively estimate the average of local values, which is important to address (C1). This is often accomplished only asymptotically, but hyperparameter tuning requires finite-time termination. Therefore, we introduce the following definition.

\begin{definition}[$(K, \varepsilon)$-average consensus]\label{def:avgcons}
A distributed algorithm $\bm{y}^{[k+1]} = \textsc{AvgCons}(\bm{y}^{[k]})$, $\bm{y}^{[k]} = (y_1^{[k]}, \ldots, y_N^{[k]})$, achieves $(K, \varepsilon)$-average consensus if
\begin{equation}\label{eq:avgcons}
    \norm{\bm{y}^{[K]} - \frac{\mathbf{1} \mathbf{1}^\top}{N} \bm{y}^{[0]}} \le \varepsilon \quad \forall i.
\end{equation}
\end{definition}

In the following we review three average consensus protocols and show under which conditions they meet Definition~\ref{def:avgcons}.

\paragraph{Standard average consensus}
The classical approach to average consensus is defined by
\begin{equation}\label{eq:linavgcons}
    y_i^{[k+1]} \;=\; w_{ii}\, y_i^{[k]} \;+\; \sum_{j \in \mathcal{N}_i} w_{ij}\, y_j^{[k]},
\end{equation}
where the $w_{ij} \ge 0$ are fixed mixing weights (collected in matrix $W$) satisfying $w_{ij} = 0$ whenever $j \notin \mathcal{N}_i \cup \{i\}$.
Algorithm~\eqref{eq:linavgcons} satisfies the following \cite{xiao_2004}.

\begin{lemma}\label{prop:stdavgcons}
If $W = [w_{ij}]$ is symmetric and doubly stochastic, then~\eqref{eq:linavgcons} achieves $(K, \varepsilon)$-average consensus with $\varepsilon = \rho^K \sigma$, where $\rho$ is the spectral radius of $W - \mathbf{1}\mathbf{1}^\top / N$ and $\sigma = \norm{(I - \mathbf{1}\mathbf{1}^\top / N) \bm{x}^{[0]}}$.
\end{lemma}

\paragraph{Finite-time exact ratio consensus (FTERC)}
The FTERC protocol~\cite{charalambous_2013, rikos_2025} achieves exact averaging in a finite number of steps, under exact arithmetic. The protocol is based on ratio consensus, and the idea is that the agents reconstruct the average (to which ratio consensus converges) after running a finite number of its iterations.

\begin{lemma}\label{prop:fterc}
FTERC, under exact arithmetic, achieves $(2 \bar{N}, 0)$-average consensus for any $\bar{N} \ge N$.
\end{lemma}

\paragraph{Finite-time quantized coordination (FTQC)}
The idea behind the FTQC protocol~\cite{bastianello_2024} is to apply distributed ADMM for consensus with quantized communications with resolution $\Delta > 0$. The following result holds%
.\footnote{Note that quantization guarantees FTQC converges in finite time when integrating a distributed termination criterion~\cite{bastianello_2024}. In this paper, however, we are interested in consensus protocols running for a fixed number of iterations, hence we fix $K$ and remove the termination criterion from FTQC.
}

\begin{lemma}\label{prop:ftqc}
FTQC achieves $(K, \varepsilon)$-average consensus with $\varepsilon = \mathcal{O}(\mu^K d_0 + \Delta)$, with $d_0$ depending on the initial conditions; in particular $\varepsilon = \mathcal{O}(\Delta)$ when $K \geq \mathcal{O}(\log(1/\Delta))$.
\end{lemma}

\paragraph{Which consensus protocol to use?}
The three consensus protocols come with different trade-offs. The standard average consensus \eqref{eq:linavgcons}, while very simple, achieves consensus only asymptotically, possibly requiring large $K$ to achieve small $\epsilon$.
FTERC achieves exact consensus in finite time, but requires exact arithmetic, which in practice means that communications need to have very high precision.
\begin{table}[ht]
\caption{Comparison of $(K, \varepsilon)$-average consensus protocols.}
\label{tab:protocols}
\begin{center}
\begin{tabular}{lcccl}
\toprule
Protocol & Parameters & Steps & $\varepsilon$ & Notes \\
\midrule
\eqref{eq:linavgcons} & $K$ & $K$ & $\rho^{K}\sigma$ & may require $K \gg 1$ \\
FTERC & $\bar{N}$ & $2\bar{N}$ & $0$ & needs exact comm. \\
FTQC & $\Delta, K$ & $K$ & $\mathcal{O}(\Delta)^{\dagger}$ & $\Delta$ bounds $\varepsilon$ \\
\bottomrule
\\[-6pt]
\end{tabular}
\end{center}
\end{table}
FTQC instead employs quantization, which reduces the communication overhead significantly; however, the resolution $\Delta$ bounds how close to the average the algorithm converges.
Therefore, each of these protocols is better suited to different application scenarios (\textit{e.g.} if there is more or less time available for running the protocol, more or less bandwidth).
Table~\ref{tab:protocols} summarizes this discussion.

\subsection{Minimum consensus}
\label{sec:mincons}
The minimum consensus protocol allows agents to cooperatively find the smallest among their initial states $y_i^{[0]} \in \mathbb{R}$, $i \in \{ 1, \ldots, N \}$. The protocol is characterized by the update
\begin{equation}\label{eq:mincons}
    y_i^{[k+1]} = \min_{j \in \mathcal{N}_i \cup \{i\}} y_j^{[k]}.
\end{equation}
On a connected graph, \eqref{eq:mincons} reaches $\min_{j \in \mathcal{V}} y_j^{[0]}$ within at most $D = \mathrm{diam}(\mathcal{G})$ iterations \cite{cortes_2008}.
However, the agents do not know the global information $D$, thus we make the following assumption.

\begin{assumption}[Upper bound on diameter]\label{ass:network}
The agents know an upper bound $\bar{D} \ge \mathrm{diam}(\mathcal{G})$.
\end{assumption}

The same protocol~\eqref{eq:mincons} applies whenever the values $y_i^{[0]}$ belong to a totally ordered set, with the minimum taken with respect to that order. This will be useful when addressing challenge (C2).

\section{Proposed Algorithms}
\label{sec:algorithm}

To address the hyperparameter selection problem~\eqref{eq:selection} we build on grid search, a classical method for hyperparameter optimization \cite{franceschi_2025}. In grid search, the search space is discretized into the product grid
\begin{equation}\label{eq:grid}
\Gamma = \Gamma_1 \times \cdots \times \Gamma_H, \qquad \Gamma_h \subset \Lambda_h,\ \lvert \Gamma_h \rvert = q_h,
\end{equation}
where $\Gamma_h$ contains equidistant points, for a total of $\lvert \Gamma \rvert = \prod_h q_h$ hyperparameter configurations.
The search evaluates $J$ over the fixed candidate set $\Gamma$ and returns the minimizer. This makes it well suited to the peer-to-peer setting, as agents can cooperatively traverse the same candidate list and select the minimizer, as discussed in the next section.
Our design will make use of the following assumption.

\begin{assumption}[Costs]\label{ass:costs}
The local costs $f_i$, $i \in \{ 1, \ldots, N \}$ are differentiable.
\end{assumption}

\subsection{Distributed grid search}
\label{sec:dgs}
The first step in our design is to define the performance metric $J$ that agents use to score candidate hyperparameters configurations. By (C1), it must be possible for the agents to cooperatively compute (or at least approximate) $J$ only using local information and peer-to-peer communications.
Since Assumption~\ref{ass:costs} holds, we can use the norm of the network-average gradient evaluated at $\bm{x}^{[\tau]} = \mathcal{A}(\bm{x}^{[0]}, \lambda, \tau)$, $\bm{x}^{[\tau]} = (x_1^{[\tau]}, \ldots, x_N^{[\tau]})$, as our metric:
\begin{equation}\label{eq:metric}
    J(\lambda) = \lVert \bar{g}(\lambda) \rVert, \qquad
    \bar{g}(\lambda) = \frac{1}{N} \sum_{i \in \mathcal{V}} \nabla f_i\bigl(x_i^{[\tau]}\bigr),
\end{equation}
where we again abuse notation by writing $\bar{g}(\lambda)$.
Note that if $\bm{x}^{[\tau]}$ is a solution of~\eqref{eq:problem} then $J(\lambda) = 0$; this will not happen in practice, since $\tau$ is finite and distributed algorithms achieve convergence asymptotically. Nonetheless, $J$ gives a good indication of how effective a configuration $\lambda$ is: all runs of $\mathcal{A}$ start from the same initial conditions $\bm{x}^{[0]}$, hence the smaller $J(\lambda)$ is, the closer $\bm{x}^{[\tau]}$ got to a stationary point over the $\tau$ iterations. Selecting for the smallest $J$ thus means that we select the configuration $\lambda$ that leads to the fastest convergence rate.

The choice of~\eqref{eq:metric} solves (C1), since each agent can compute $g_i = \nabla f_i(x_i^{[\tau]})$ locally after $\mathcal{A}$ has run, and then apply $K$ iterations of a consensus protocol to compute (possibly approximately) their average $g(\lambda)$. Thus these computations only involve local information and peer-to-peer communications. Applying this idea leads to the Distributed Grid Search (D-GS) described in Algorithm~\ref{alg:dgs}.
\begin{algorithm2e}[ht]
\caption{Distributed Grid Search (D-GS)}
\label{alg:dgs}
\Input{grid $\Gamma = \{\lambda^{(1)}, \dots, \lambda^{(\lvert \Gamma \rvert)}\}$ of candidates $\lambda^{(\ell)} \in \Lambda$, horizon $\tau$, initial points $\bm{x}^{[0]}$, algorithm $\mathcal{A}$, averaging protocol \textsc{AvgCons}}
$(J_i^\star, \lambda_i^\star) \gets (+\infty, \varnothing)$\;
\For{$\ell = 1, \dots, \lvert \Gamma \rvert$, each agent $i \in \mathcal{V}$\label{line:dgs-loop}}{
    $x_i^{[\tau]} \gets \mathcal{A}(\bm{x}^{[0]}, \lambda^{(\ell)}, \tau)$\label{line:dgs-run}\;
    $g_i \gets \nabla f_i(x_i^{[\tau]})$\label{line:dgs-grad}\;
    $\bar g_i \gets \textsc{AvgCons}(\bm{g})$\label{line:dgs-cons}%
        \tcp*[l]{coordination, $\bar g_i \approx \bar{g}(\lambda^{(\ell)})$}
    $J_i \gets \lVert \bar g_i \rVert$\label{line:dgs-metric}\;%
    %\tcp*[l]{$\lambda^{(\ell)} \notin \Gamma_{\mathrm{fin}}$}
    \If{$J_i \le J_i^\star$\label{line:dgs-best}}{
        $(J_i^\star, \lambda_i^\star) \gets (J_i, \lambda^{(\ell)})$\;
    }
}
$(\hat{J}, \hat{\lambda}) \gets \textsc{MinCons}\big(\bm{J}^\star, \bm{\lambda}^\star\big)$\label{line:dgs-min}%
    \tcp*[l]{agreement}
\end{algorithm2e}

In D-GS, the agents run $\mathcal{A}$ for $\tau$ iterations with configuration $\lambda^{(\ell)}$, then compute the local gradient, and approximate their average through consensus, which yields $J_i$. Based on the value of $J_i$, each agent \textit{locally} decides whether $\lambda^{(\ell)}$ is better than the previous best configuration; that is, if it led to faster convergence.
Importantly, depending on the consensus protocol that is employed, $\{ J_i \}_{i = 1}^N$ might be different; that is, they only approximate the average (this is the case of~\eqref{eq:linavgcons} and FTQC). As a consequence, it is possible that only some of the agents update their local candidate $\lambda_i^\star$.
For this reason, the algorithm integrates an agreement routine at line~\ref{line:dgs-min}. The idea is for the agents to apply $\bar{D}$ minimum consensus iterations on the pairs $\{ (J_i^\star, \lambda_i^\star) \}_{i = 1}^N$ using the following ordering: $(J(\lambda^{(\ell)}), \lambda^{(\ell)}) < (J(\lambda^{(k)}), \lambda^{(k)})$ either if $J(\lambda^{(\ell)}) < J(\lambda^{(k)})$, or if $J(\lambda^{(\ell)}) = J(\lambda^{(k)})$ and $\ell > k$. This means that candidate configurations with smaller $J$ are prioritized, and ties are broken by selecting lexicographically larger configurations. This is a heuristic motivated by the fact that larger stepsizes are preferable, but alternative rules are possible.
Since by Assumption~\ref{ass:network} $\bar{D}$ is an upper bound of the diameter, minimum consensus is guaranteed to terminate with all agents sharing the same configuration, thus solving challenge (C2).

\begin{remark}[Divergent configurations]
Depending on the selected search space $\Lambda$, some of the configurations that are tested by the agents might lead to divergence of $\mathcal{A}$. These would lead to very large values of $J$ and are thus weeded out. In practical implementations, $J$ for a divergent run of $\mathcal{A}$ might return infinity or not-a-number; the latter causes issues when comparing configurations, so in our implementation the agents replace it by $\infty$.
\end{remark}

The above remark highlights the fact that so far we have not constrained the search space in any way. One might select the theoretical hyperparameter ranges defined for $\mathcal{A}$, but, especially for stepsizes, these are extremely conservative (see section~\ref{sec:res-cost}). Rather, hyperparameter tuning is more beneficial when a wider space is explored, finding configurations that do lead to convergence while improving significantly on the performance of theoretical choices.
This motivates the following assumption that will hold throughout.

\begin{assumption}[Search space]\label{ass:feasible-search}
The search grid $\Gamma$ includes at least one configuration for which $\mathcal{A}$ converges.
\end{assumption}

This assumption is very mild and easily satisfied by having the theoretical range of hyperparameters that guarantee convergence be a subset of the search space.

\begin{remark}[Common knowledge]
The set-up of Algorithm~\ref{alg:dgs} requires a degree of common knowledge among the agents, including: the algorithm to be run ($\mathcal{A}$, $\tau$, $\bm{x}^{[0]}$), the search grid and the average consensus protocol. This is no different from all other distributed set-ups where the agents need to know what algorithm to run. Additionally, this information can easily be agreed upon in a distributed fashion.
\end{remark}

\subsection{Distributed zoom grid search}
\label{sec:dzgs}
While D-GS enables distributed hyperparameter tuning, this method comes with following drawbacks: 1) the computational cost grows exponentially with the number of points in the grid (since $\lvert \Gamma \rvert = \prod_k q_k$), and 2) its accuracy is bounded by the resolution of the grid.
To address these challenges, we leverage \textit{zoom grid search} \cite{franceschi_2025}, which applies grid search sequentially over progressively narrower search spaces.

Starting from the search space $\Lambda^{(1)} = \Lambda$, at each zoom round $r$ the agents construct a new, narrower search space $\Lambda^{(r)} \subset \Lambda^{(r-1)} \subset \ldots \subset \Lambda^{(1)}$ based on the results of the previous hyperparameters evaluations.
Formally, let $(\hat{J}^{(r)}, \hat{\lambda}^{(r)})$ be the best configuration found by D-GS during round $r$ in the grid $\Gamma^{(r)} \subset \Lambda^{(r)}$; the new search space is defined as:
\begin{equation}\label{eq:zoom}
    \Lambda_k^{(r+1)} = \Bigl[\hat{\lambda}_k^{(r)} - \tfrac{h_k^{(r)}}{2},\; \hat{\lambda}_k^{(r)} + \tfrac{h_k^{(r)}}{2}\Bigr] \cap \Lambda_k,
\end{equation}
where $h_k^{(r)} = \gamma\, \mathrm{diam}(\Lambda_k^{(r)})$ is the contracted width, $\gamma \in (0,1)$ the contraction factor. The agents can then define the new grid $\Gamma^{(r+1)}$ and apply D-GS to it.
\begin{algorithm2e}[ht]
\caption{Distr. Zoom Grid Search~(D-ZGS)}
\label{alg:dzgs}
\Input{search space $\Lambda$, grid sizes $q_1, \dots, q_K$, rounds $R$, contraction $\gamma$, D-GS}
$\Lambda^{(1)} \gets \Lambda$, \quad $(\hat{J}, \hat{\lambda}) \gets (+\infty, \varnothing)$\;
\For{$r = 1, \dots, R$, each agent $i \in \mathcal{V}$\label{line:dzgs-loop}}{
    build $\Gamma^{(r)}$ over $\Lambda^{(r)}$ according to \eqref{eq:grid}\label{line:dzgs-grid}\;
    $(\hat{J}^{(r)}, \hat{\lambda}^{(r)}) \gets \textsc{D-GS}(\Gamma^{(r)}, \tau, \bm{x}^{[0]}, \mathcal{A}, \textsc{AvgCons})$\label{line:dzgs-dgs}\;
    \If{$r = 1$ \textbf{or} $\hat{J}^{(r)} < \hat{J}$\label{line:dzgs-check}}{
        $(\hat{J}, \hat{\lambda}) \gets (\hat{J}^{(r)}, \hat{\lambda}^{(r)})$\;
    }
    update $\Lambda^{(r+1)}$ according to \eqref{eq:zoom}\label{line:dzgs-zoom}%
    \tcp*[l]{zoom}
}
\end{algorithm2e}
The resulting algorithm is described in Algorithm~\ref{alg:dzgs}, and illustrated in Figure~\ref{fig:zoom-depiction}.
\begin{figure}[ht]
    \centering
    \includegraphics[width=0.9\linewidth]{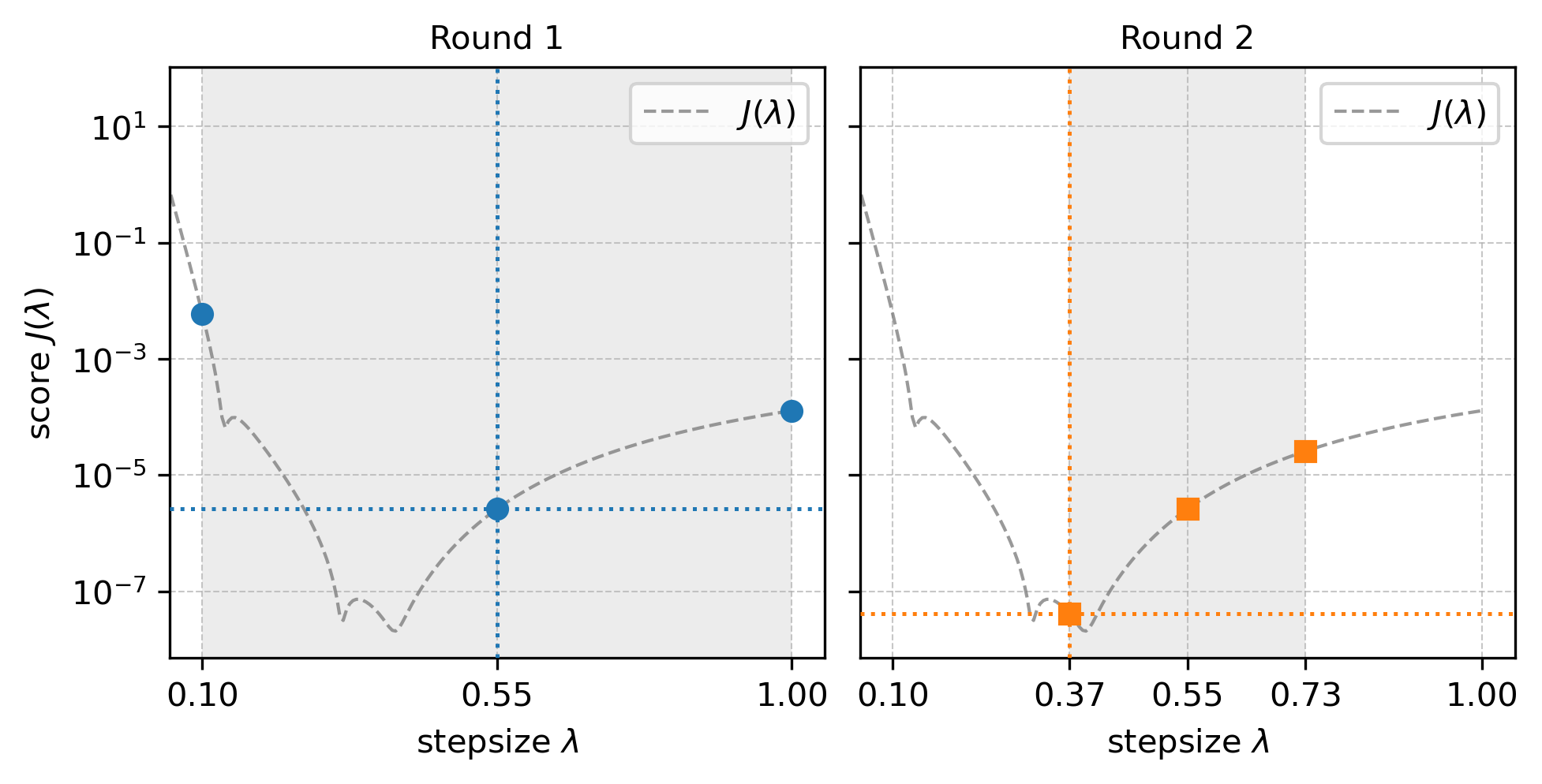}
    \caption{Depiction of D-ZGS with two rounds (in the numerical set-up of section~\ref{sec:res-cost}); the shaded region is $\Lambda^{(r)}$ and the dots the points in $\Gamma^{(r)}$.}
    \label{fig:zoom-depiction}
\end{figure}
We remark that D-ZGS does not require any additional coordination than D-GS, as the zoomed-in search space can be constructed independently by all agents once D-GS outputs the candidate configuration. A drawback of D-ZGS is that it can lead to suboptimal results, as its greedy zooming strategy may prematurely exclude regions containing better yet unexplored configurations.

\subsection{Algorithms' analysis}
\label{sec:analysis}
We now characterize the (sub)optimality of the hyperparameters configuration produced by D-GS and D-ZGS.

\begin{proposition}[D-GS]\label{thm:main}
Let Assumption~\ref{ass:feasible-search} hold and let \textsc{AvgCons} satisfy Definition~\ref{def:avgcons}.
Then D-GS outputs the configuration $\hat{\lambda}$ in finite time, and it satisfies:
\begin{equation}\label{eq:J_bound}
    \min_{\lambda \in \Gamma} J(\lambda) \leq J(\hat{\lambda}) \;\le\; \min_{\lambda \in \Gamma} J(\lambda) + 2 \varepsilon.
\end{equation}
\end{proposition}
\begin{proof}
Let us first prove that D-GS terminates in finite time.
For each candidate $\lambda \in \Gamma$, the agents collectively run $\tau$ iterations of $\mathcal{A}$; they then evaluate the local gradients, and run \textsc{AvgCons} for $K$ iterations to achieve Definition~\ref{def:avgcons}.
After all the candidate configurations have been tested, the agents run $\bar{D}$ iterations of \textsc{MinCons}, and successfully select $\hat{\lambda}$.
All of these steps take a finite amount of time, guaranteeing termination in finite time.

We turn now to proving~\eqref{eq:J_bound}.
Let us denote $\bar{g}_i(\lambda)$ the output of \textsc{AvgCons} for agent $i$ when configuration $\lambda$ is being evaluated; and $J_i(\lambda) = \norm{\bar{g}_i(\lambda)}$.
Since \textsc{AvgCons} satisfies Definition~\ref{def:avgcons}, letting $\mathbf{\bar{g}}(\lambda) = (\bar{g}_1(\lambda), \ldots, \bar{g}_N(\lambda)$, for any configuration $\lambda$ we have
\begin{align*}
    \varepsilon^2 &\geq \norm{\mathbf{\bar{g}}(\lambda) - \bm{1} \otimes \bar{g}(\lambda)}^2 = \sum_{i \in \mathcal{V}} \norm{\bar{g}_i(\lambda) - \bar{g}(\lambda)}^2 \\
    &\geq \norm{\bar{g}_i(\lambda) - \bar{g}(\lambda)}^2 \quad \forall i \in \mathcal{V}
\end{align*}
and hence $\norm{\bar{g}_i(\lambda) - \bar{g}(\lambda)} \leq \varepsilon$.
Using the reverse triangle inequality we can then write
\begin{align*}
    \varepsilon \geq \norm{\bar{g}_i(\lambda) - \bar{g}(\lambda)} \geq \left| \norm{\bar{g}_i(\lambda)} - \bar{g}(\lambda) \right| = \left| J_i(\lambda) - J(\lambda) \right|.
\end{align*}
As a consequence, $J_i(\lambda) \leq J(\lambda) + \varepsilon$ and $J(\lambda) \leq J_i(\lambda) + \varepsilon$.

We are now ready to prove~\eqref{eq:J_bound}. Let $\lambda^\star \in \arg\min_{\lambda \in \Gamma} J(\lambda)$, and $i \in \mathcal{V}$ be such that $\hat{\lambda} = \lambda_i^\star$ (that is, the configuration $\lambda_i^\star$ proposed by $i$ is selected during \textsc{MinCons}); then we have:
\begin{align*}
    J(\lambda^\star) &\overset{(i)}{\leq} J(\hat{\lambda}) \overset{(ii)}{=} J(\lambda_i^\star) \overset{(iii)}{\leq} J_i(\lambda_i^\star) + \varepsilon \overset{(iv)}{=} \hat{J} + \varepsilon \\
    &\overset{(v)}{\leq} J_i(\lambda^\star) + \varepsilon \overset{(vi)}{\leq} J(\lambda^\star) + 2 \varepsilon,
\end{align*}
where $(i)$ holds by definition of $\lambda^\star$; $(ii)$ by the fact that $i$'s configuration is selected; $(iii)$ by $J(\lambda) \leq J_i(\lambda) + \varepsilon$ proved above for any $i \in \mathcal{V}$ and $\lambda \in \Gamma$; $(iv)$ same as $(ii)$; $(v)$ since, for any $\lambda \in \Gamma$, $J_i(\lambda_i^\star) \leq J_i(\lambda)$ by construction; $(vi)$ by $J_i(\lambda) \leq J(\lambda) + \varepsilon$ proved above for any $i \in \mathcal{V}$ and $\lambda \in \Gamma$.
\end{proof}

\smallskip

Clearly, the more accurate \textsc{AvgCons} is, the closer the chosen configuration $\hat{\lambda}$ is to the optimal one in $\Gamma$. Now, since D-ZGS employs D-GS as a subroutine, we can state the following result as a consequence of Proposition~\ref{thm:main}.

\begin{corollary}[D-ZGS]\label{cor:dzgs}
Let Assumption~\ref{ass:feasible-search} hold and let \textsc{AvgCons} satisfy Definition~\ref{def:avgcons}.
Then D-ZGS outputs the configuration $\hat{\lambda}$ in finite time, and satisfies:
\begin{equation}\label{eq:dzgs_bound}
    \min_{\lambda \in \Gamma^{(r^\star)}} J(\lambda) \leq J(\hat{\lambda}) \;\le\; \min_{\lambda \in \Gamma^{(r^\star)}} J(\lambda) + 2 \varepsilon.
\end{equation}
where $r^\star$ is the round during which $\hat{\lambda}$ is selected.
\end{corollary}

We conclude this section by evaluating the computational complexity of the proposed algorithms. Let $t_{\mathcal{A}}$ and $t_{\mathrm{comm}}$ be the time (in seconds) required to run one iteration of $\mathcal{A}$ and consensus, respectively. Then D-GS and D-ZGS terminate after $|\Gamma| \tau t_{\mathcal{A}} + (|\Gamma| K + \bar{D}) t_{\mathrm{comm}}$ and $\sum_{r = 1}^R |\Gamma^{(r)}| \tau t_{\mathcal{A}} + (|\Gamma^{(r)}| K + \bar{D}) t_{\mathrm{comm}}$ seconds.

\section{Numerical Results}
\label{sec:results}
In this section we evaluate the proposed methods on optimization problems of the form~\eqref{eq:problem} with linear and logistic regression local costs, defined by
$
f_i(x) = \tfrac{1}{2M} \lVert A_i x - b_i \rVert^2,
$
with $A_i \in \mathbb{R}^{M \times n}$ the local data matrix and $b_i \in \mathbb{R}^{M}$ the local targets, and
$
f_i(x) = \frac{1}{M} \sum_{m=1}^{M} \log\bigl(1 + \exp\bigl(-b_{i,m}\, a_{i,m}^{\top} x\bigr)\bigr),
$
with samples $a_{i,m} \in \mathbb{R}^{n}$ and labels $b_{i,m} \in \{-1, +1\}$, respectively.
The $N = 20$ agents hold each $M = 30$ samples of dimension $n = 4$, and communicate over a connected $3$-regular graph.
We use $\tau = 50$ and, unless otherwise stated, average consensus~\eqref{eq:linavgcons} with Metropolis weights and $K = 600$.
The data are randomly generated, and the results are averaged over $30$ instances of each problem.
The experiments are implemented in \texttt{decent-bench} \cite{decent_bench}.

\subsection{Comparison with centralized reference}
\label{sec:res-central}
We tune three different algorithms, namely, Exact Diffusion (ED)~\cite{yuan_2019}, distributed ADMM~\cite{bastianello_2018}, and ADMM-Tracking Gradient (ATG)~\cite{carnevale_2023}. These algorithms have $1$, $2$, and $4$ hyperparameters, respectively.
The grids are uniform, and D-ZGS runs $R = 2$ rounds with $\gamma = 0.2$ on coarser grids, so that its total number of candidates is at most half that of D-GS.
The quality of the configuration returned by D-GS or D-ZGS is measured by computing the empirical convergence rate $\nu$ of $\lVert\nabla F(\bar x^t)\rVert$, with smaller $\nu$ meaning faster convergence.
As a reference we use centralized grid search (C-GS), which has access to all agents' data and scores the same grid as D-GS directly using $\nu$ as a metric.
Table~\ref{tab:main} reports the mean rates over the $30$ problem instances.
\begin{table}[ht]
\caption{Mean convergence rate under tuned hyperparameters.}
\label{tab:main}
\begin{center}
\begin{tabular}{llccc}
\toprule
Problem & Algorithm & D-GS & D-ZGS & C-GS \\
\midrule
\multirow{3}{*}{Linear}
 & ED   & 0.88 & 0.88 & 0.87 \\
 & ADMM & 0.76 & 0.76 & 0.75 \\
 & ATG  & 0.79 & 0.78 & 0.79 \\
\midrule
\multirow{3}{*}{Logistic}
 & ED   & 0.89 & 0.89 & 0.88 \\
 & ADMM & 0.78 & 0.75 & 0.74 \\
 & ATG  & 0.82 & 0.83 & 0.81 \\
\bottomrule
\end{tabular}
\end{center}
\end{table}
In five of the six cases D-GS matches C-GS to within $0.01$, and D-ZGS matches or improves on D-GS, despite using a reduced computational budget.
We observe a larger performance gap for ADMM applied to logistic costs; in this case, the agents perform approximate updates, and fail to reach optimality. Therefore, minimizing $\nu$ (as in C-GS) might not be the best strategy and minimizing the gradient norm (as in D-GS and D-ZGS) might yield lower suboptimality. Indeed, the mean gradient norm achieved by ADMM under the different configurations is: $2.90 \times 10^{-6}$ for C-GS, $2.20 \times 10^{-6}$ for D-GS, and $1.21 \times 10^{-6}$ for D-ZGS.

\subsection{Comparison with theoretical hyperparameters}
\label{sec:res-cost}
We compare the computational cost of distributed tuning and deployment of the algorithm against deployment with hyperparameters tuned according to theoretical results.
We consider ED for linear regression, using on the one hand the maximum stepsize allowed by the results in \cite{yuan_2019a} (which we remark requires global knowledge), and, on the other hand, D-GS with search space $(0,\, 2/L]$, where $L = \max_i L_i$ and $L_i$ the smoothness constant of $f_i$.

In Table~\ref{tab:cost}, we report the computational cost incurred by ED with theoretical stepsize, and D-GS with different sizes of grid $\Gamma$. The cost is computed as the sum of the total number of gradient evaluations and peer-to-peer communications, accounting for both tuning and deployment for D-GS.
\begin{table}[ht]
\caption{Computational cost, in units of $10^{3}$.}
\label{tab:cost}
\begin{center}
\begin{tabular}{lcccc}
\toprule
Method & $\lvert \Gamma \rvert$ & Tuning & Deployment & Total \\
\midrule
Theoretical bound & --- & 0 & 16,681 & 16,681 \\
\midrule
\multirow{6}{*}{D-GS}
& 3 & 5.4 & 16.8 & 22.3 \\
& 4 & 7.2 & 5.6 & 12.9 \\
& 5 & 9.0 & 4.0 & 13.0 \\
& 10 & 18.1 & 3.6 & 21.7 \\
& 20 & 36.1 & 3.6 & 39.8 \\
& 100 & 181 & 3.6 & 184 \\
\bottomrule
\end{tabular}
\end{center}
\end{table}
We observe that the theoretical bound is extremely conservative, leading to very poor performance, whereas distributed tuning allows finding a better configuration by exploring a wider search space. We also notice that denser grids lead to better results at the price of a more expensive tuning phase.
Figure~\ref{fig:cost} depicts an example of deployment for different configurations.
\begin{figure}[ht]
    \centering
    \includegraphics[width=0.75\linewidth]{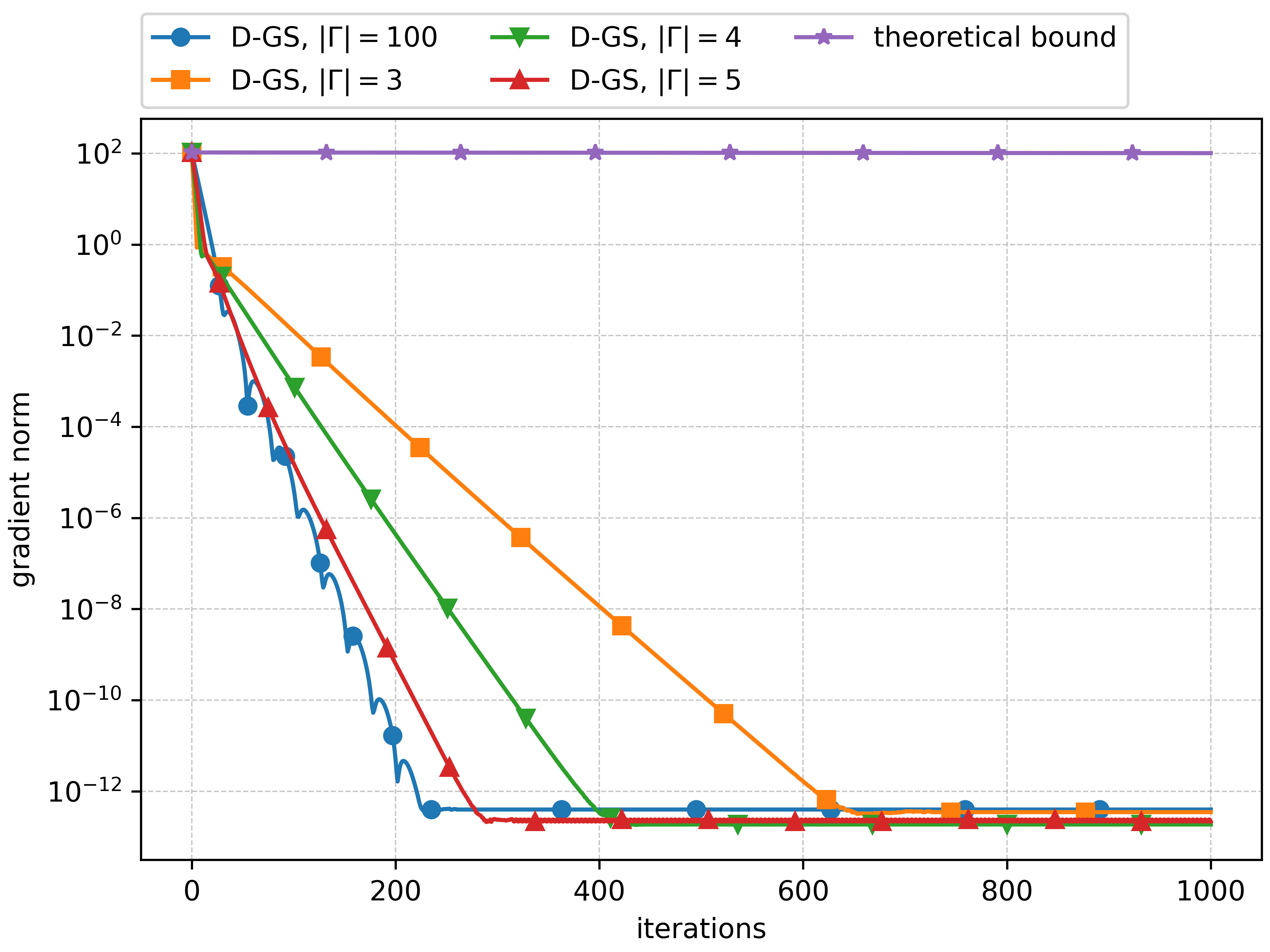}
    \caption{Trajectory of the gradient norm under selected configurations.}
    \label{fig:cost}
\end{figure}

\subsection{Selecting the average consensus protocol}
\label{sec:res-robust}
The previous results used a large $K$ so that consensus error was negligible. In this section we focus on the consensus module employed by D-GS, comparing the alternatives in Table~\ref{tab:protocols}.
To this end, we artificially perturb the average consensus~\eqref{eq:linavgcons} in D-GS to achieve error $\varepsilon$ in mean, and evaluate the performance of the resulting configuration tuned by the algorithm, as depicted in Figure~\ref{fig:robustness}.
\begin{figure}[ht]
    \centering
    \includegraphics[width=0.75\linewidth]{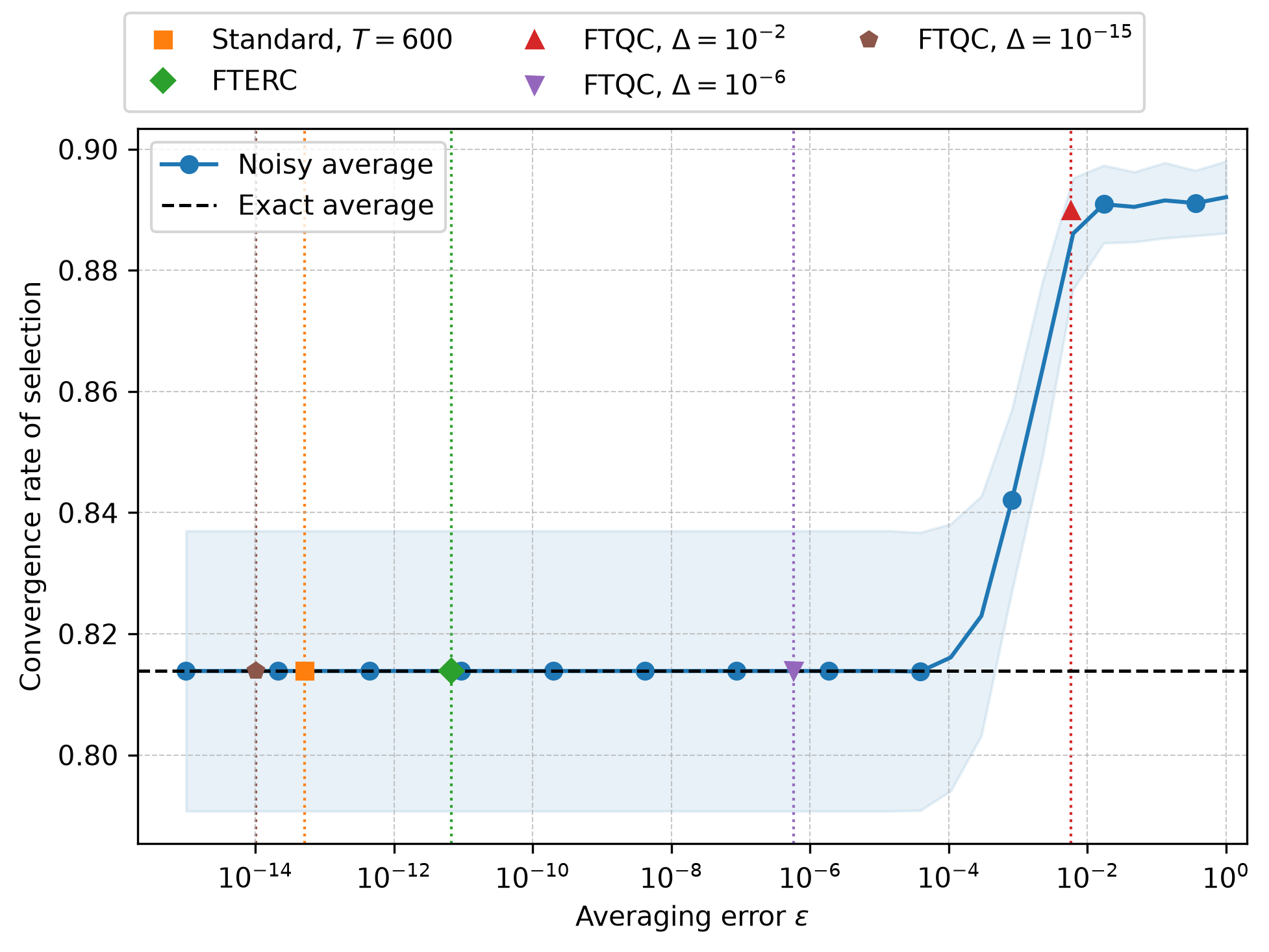}
    \caption{Convergence rate of the configuration selected by D-GS, as a function of the averaging error $\varepsilon$, with the protocols of Table~\ref{tab:protocols}.}
    \label{fig:robustness}
\end{figure}
Interestingly, the impact of the consensus error on performance is felt only above $10^{-4}$; above that threshold, the tuning becomes less accurate. The figure also marks where the protocols of Table~\ref{tab:protocols} lie in terms of performance; this shows, for example, that FTQC with moderate quantization performs well, while reducing the communication overhead.

\section{Conclusion}
\label{sec:conclusion}
We proposed two distributed hyperparameter tuning algorithms, D-GS and D-ZGS, which rely only on local computations and peer-to-peer communications. We compared the computational complexity of the two methods and evaluated the (sub)optimality of the hyperparameters configurations they tune. We evaluated the algorithms on linear and logistic regression problems, comparing against a centralized best choice and a choice dictated by theoretical analyses.
Future work will build on this framework, introducing asynchronous termination to grid search to cut tuning time and expanding to different search strategies.

\bibliographystyle{IEEEtran}
\bibliography{references}

\end{document}